\documentclass[12pt]{article}    

\usepackage[margin=0.75in]{geometry} 
\usepackage{amsmath,amssymb,amsthm}
\usepackage{color}

\newtheorem{theorem}{Theorem}[section]
\newtheorem{proposition}[theorem]{Proposition}
\newtheorem{lemma}[theorem]{Lemma}

\newtheorem{remark}{Remark}

\newcommand{\al}{\alpha}
\newcommand{\bt}{\beta}

\newcommand{\s}{\sigma}

\newcommand{\be}{\begin{equation}}
\newcommand{\ee}{\end{equation}}
\newcommand{\bea}{\begin{eqnarray}}
\newcommand{\eea}{\end{eqnarray}}

\numberwithin{equation}{section}
\begin{document}

\title{Semi-classical Jacobi Polynomials, Hankel Determinants and Asymptotics}
\author{Chao Min\thanks{School of Mathematical Sciences, Huaqiao University, Quanzhou 362021, China; e-mail: chaomin@hqu.edu.cn}\: and Yang Chen\thanks{Department of Mathematics, Faculty of Science and Technology, University of Macau, Macau, China; e-mail: yangbrookchen@yahoo.co.uk}}


\date{\today}
\maketitle
\begin{abstract}
We study orthogonal polynomials and Hankel determinants generated by a symmetric semi-classical Jacobi weight.
By using the ladder operator technique, we derive the second-order nonlinear difference equations satisfied by the recurrence coefficient $\bt_n(t)$ and the sub-leading coefficient $\mathrm{p}(n,t)$ of the monic orthogonal polynomials. This enables us to obtain the large $n$ asymptotics of $\bt_n(t)$ and $\mathrm{p}(n,t)$ based on the result of Kuijlaars et al. [Adv. Math. \textbf{188} (2004) 337-398]. In addition, we show the second-order differential equation satisfied by the orthogonal polynomials, with all the coefficients expressed in terms of $\bt_n(t)$.

From the $t$ evolution of the auxiliary quantities, we prove that $\bt_n(t)$ satisfies a second-order differential equation and $R_n(t)=2n+1+2\al-2t(\bt_n(t)+\bt_{n+1}(t))$ satisfies a particular Painlev\'{e} V equation under a simple transformation. Furthermore, we show that the logarithmic derivative of the associated Hankel determinant satisfies both the second-order differential and difference equations. The large $n$ asymptotics of the Hankel determinant is derived from its integral representation in terms of $\bt_n(t)$ and $\mathrm{p}(n,t)$.
\end{abstract}

$\mathbf{Keywords}$: Semi-classical Jacobi polynomials; Hankel determinants; Ladder operators;

Painlev\'{e} V; Differential and difference equations; Asymptotic expansions.

$\mathbf{Mathematics\:\: Subject\:\: Classification\:\: 2020}$: 42C05, 41A60, 33E17.

\section{Introduction}
The relationship between semi-classical orthogonal polynomials and Painlev\'{e} equations has been studied extensively over the past few years. For example, Chen and Its \cite{ChenIts} considered the singularly perturbed Laguerre weight $w(x)=x^{\al}\mathrm{e}^{-x-s/x},\;x\geq 0, \al>0, s>0.$ They showed that the diagonal recurrence coefficient of the monic orthogonal polynomials satisfies a particular Painlev\'{e} III, and the logarithmic derivative of the associated Hankel determinant satisfies the Jimbo-Miwa-Okamoto (JMO) $\s$-form of the Painlev\'{e} III. Filipuk, Van Assche and Zhang \cite{Filipuk} investigated another semi-classical Laguerre weight $w(x)=x^{\al}\mathrm{e}^{-x^2+tx},\;x\in\mathbb{R}^{+}, \al>-1, t\in \mathbb{R}$, to establish the relation between the recurrence coefficients of the orthonormal polynomials and the Painlev\'{e} IV. Later, Clarkson and Jordaan \cite{Clarkson} proved that the logarithmic derivative of the associated Hankel determinant satisfies the JMO $\s$-form of the Painlev\'{e} IV. See also \cite{BCE,Clarkson3,Dai,Min2020b,MLC} for more information.

In this paper, we are concerned with the symmetric semi-classical Jacobi weight
\be\label{wei}
w(x,t):=(1-x^2)^\al\mathrm{e}^{-tx^2},\qquad x\in[-1,1],\;\;\al>-1,\;\;t\in \mathbb{R}.
\ee
Let $P_{n}(x,t),\; n=0,1,2,\ldots,$ be the monic polynomials of degree $n$ orthogonal with respect to the weight (\ref{wei}), i.e.,
\be\label{or}
\int_{-1}^{1}P_{m}(x,t)P_{n}(x,t)w(x,t)dx=h_{n}(t)\delta_{mn},\qquad m, n=0,1,2,\ldots.
\ee
Since the weight $w(x,t)$ is even, $P_{n}(x,t)$ has the following monomial expansion \cite[p. 21]{Chihara},
\be\label{expan}
P_{n}(x,t)=x^{n}+\mathrm{p}(n,t)x^{n-2}+\cdots,\qquad n=0,1,2,\ldots,
\ee
where $\mathrm{p}(n,t)$ denotes the coefficient of $x^{n-2}$, and the initial values of $\mathrm{p}(n,t)$ are defined to be $\mathrm{p}(0,t)=0,\; \mathrm{p}(1,t)=0$.

The orthogonal polynomials $P_{n}(x,t),\; n=0,1,2,\ldots,$ satisfy the following three-term recurrence relation \cite[p. 18-21]{Chihara}
\be\label{rr}
xP_{n}(x,t)=P_{n+1}(x,t)+\beta_{n}(t)P_{n-1}(x,t),
\ee
with the initial conditions
$$
P_{0}(x,t)=1,\;\;\;\;P_{-1}(x,t)=0.
$$

The combination of (\ref{or}), (\ref{expan}) and (\ref{rr}) shows that the recurrence coefficient $\bt_n(t)$ has two alternative representations:
\bea
\beta_{n}(t)&=&\mathrm{p}(n,t)-\mathrm{p}(n+1,t)\label{be1}\\
&=&\frac{h_{n}(t)}{h_{n-1}(t)}.\label{be2}
\eea
Taking a telescopic sum of (\ref{be1}), we get an important identity
\be\label{sum}
\sum_{j=0}^{n-1}\beta_{j}(t)=-\mathrm{p}(n,t).
\ee
From (\ref{be1}) we also have $\bt_0(t)=0$.

We introduce the Hankel determinant generated by the weight (\ref{wei}),
$$
D_n(t)=\det\left(\mu_{j+k}(t)\right)_{j,k=0}^{n-1},
$$
where $\mu_{j}(t),\;j=0,1,2,\ldots$ are the moments given by
$$
\mu_{j}(t)=\int_{-1}^{1}x^{j}w(x,t)dx.
$$
The moments $\mu_{j}(t)$ can be evaluated explicitly:
\be\label{mom}
\mu_{j}(t)=\frac{[1+(-1)^j]\:\Gamma(1+\al)\:\Gamma\big(\frac{j+1}{2}\big)\:\Phi\big(\frac{j+1}{2},\frac{j+3}{2}+\al;-t\big)}{2\:\Gamma\big(\frac{j+3}{2}+\al\big)},\qquad j=0,1,2,\ldots,
\ee
where $\Gamma(\cdot)$ is the gamma function defined by $\Gamma(z)=\int_{0}^{\infty}\mathrm{e}^{-x}x^{z-1}dx,\;\mathrm{Re}\; z>0$, and $\Phi(\cdot, \cdot; \cdot)$ is the Kummer's confluent hypergeometric function defined by \cite[p. 260]{Lebedev}
$$
\Phi(a, b; z)=\sum_{k=0}^{\infty}\frac{(a)_{k}}{(b)_{k}}\frac{z^{k}}{k!},\qquad |z|<\infty,\quad b\neq 0, -1, -2,\ldots
$$
and the symbol $(a)_{k}$ denotes the quantity
$$
(a)_{0}=1,\qquad (a)_{k}=\frac{\Gamma(a+k)}{\Gamma(a)}=a(a+1)\cdots(a+k-1),\quad k=1,2,\ldots.
$$
Furthermore, $\Phi(a, b; z)$ has an integral representation \cite[p. 266]{Lebedev}
$$
\Phi(a, b; z)=\frac{\Gamma(b)}{\Gamma(a)\Gamma(b-a)}\int_{0}^{1}\mathrm{e}^{zx}x^{a-1}(1-x)^{b-a-1}dx,\quad \mathrm{Re}\; b>\mathrm{Re}\; a>0.
$$
It is a known fact that $D_n(t)$ can be expressed as the product of $h_j(t)$ \cite[(2.1.6)]{Ismail},
\be\label{hankel}
D_{n}(t)=\prod_{j=0}^{n-1}h_{j}(t).
\ee
In addition, from (\ref{be2}) and (\ref{hankel}) we have the following relation:
$$
\bt_n(t)=\frac{D_{n+1}(t)D_{n-1}(t)}{D_n^2(t)}.
$$

Hankel determinants generated by the perturbed Gaussian, Laguerre and Jacobi weights have been studied a lot in recent years \cite{Basor2015,BCI,CD,Min2020,Min2020b,MLC,Xu2015,Zeng}. This is in part due to the connection with random matrix theory, since Hankel determinants may represent the partition functions of the perturbed unitary ensembles.
Hankel determinants are also closely related to the study of orthogonal polynomials and its recurrence coefficients \cite{Ismail,Szego}.

The study of the asymptotics of the perturbed Hankel determinants is of particular interest. For example, Kuijlaars et al. \cite{Kuij} considered a modified Jacobi weight
$
w(x)=(1-x)^\al(1+x)^\bt h(x),\; x\in[-1,1], \al, \bt>-1,
$
where $h(x)$ is real analytic and strictly positive on $[-1,1]$. They obtained the large $n$ asymptotics of the orthogonal polynomials, the recurrence coefficients and the associated Hankel determinant by using the Riemann-Hilbert approach. However, there is an undetermined constant in the asymptotics of the Hankel determinant \cite[Corollary 1.8]{Kuij}. This constant is obtained by Charlier and Gharakhloo \cite[Theorem 1.3]{CG} who investigated a more general weight. In addition, one can also refer to the study of the asymptotics of other Hankel determinants with jump discontinuities \cite{BCI,Its,Min2019} and with Fisher-Hartwig singularities \cite{Charlier,CD,DIK,Min2020,Wu}.


Hankel determinants play a fundamental role in random matrix theory \cite{Mehta}. It is well known that the Hankel determinant $D_n(t)$ can be viewed as the partition function of the perturbed Jacobi unitary ensemble \cite[Corollary 2.1.3]{Ismail}, i.e.,
$$
D_n(t)=\frac{1}{n!}\int_{[-1,1]^n}\prod_{1\leq i<j\leq n}(x_i-x_j)^2\prod_{k=1}^n w(x_k,t)dx_k.
$$
When $t=0$, (\ref{wei}) is the standard symmetric Jacobi weight, and we have the following formula for $D_n(0)$ \cite[(17.6.2)]{Mehta}:
\bea\label{bg}
D_n(0)&=&\frac{2^{n(n+2\al)}}{n!}\prod_{j=1}^{n}\frac{\Gamma(j+1)\Gamma^2(j+\al)}{\Gamma(j+n+2\al)}\nonumber\\
&=&2^{n(n+2\al)}\frac{G(n+1)G(n+1+2\al)G^2(n+\al+1)}{G(2n+2\al+1)G^2(\al+1)},
\eea
where $G(\cdot)$ is the Barnes $G$-function which satisfies the relation \cite{Barnes,Voros}
$$
G(z+1)=\Gamma(z)G(z),\qquad\qquad G(1):=1.
$$

The rest of the paper is divided into sections as follows. In the next section, we apply the ladder operators and its supplementary conditions to the monic orthogonal polynomials with respect to the semi-classical Jacobi weight. This allows us to obtain the second-order difference equations satisfied by $\bt_n(t)$ and $\mathrm{p}(n,t)$, respectively. We also derive the second-order differential equation for the orthogonal polynomials, with the coefficients all expressed in terms of $\bt_n(t)$. In Section 3, we give the large $n$ asymptotic expansions of $\bt_n(t)$ and $\mathrm{p}(n,t)$ based on the difference equations. We consider the $t$ evolution and establish the relation of our problem with the Painlev\'{e} equations in Section 4. In the last section, we derive the second-order differential and difference equations satisfied by a quantity related to the logarithmic derivative of the Hankel determinant. Finally, we obtain the large $n$ asymptotic expansion of our Hankel determinant.

\section{Ladder Operators and Second-Order Nonlinear Difference Equations}
The ladder operator approach has been widely applied to the study of orthogonal polynomials and Hankel determinants for the perturbed Gaussian, Laguerre and Jacobi weights; see, e.g., \cite{BCE,Basor2015,ChenIts,Dai,Min2019,Min2020b,MLC}. Following Basor et al. \cite{BCE} and Chen and Its \cite{ChenIts}, we have a pair of ladder operators for our semi-classical Jacobi polynomials:
\be\label{lowering}
\left(\frac{d}{dz}+B_{n}(z)\right)P_{n}(z)=\beta_{n}A_{n}(z)P_{n-1}(z),
\ee
\be\label{raising}
\left(\frac{d}{dz}-B_{n}(z)-\mathrm{v}'(z)\right)P_{n-1}(z)=-A_{n-1}(z)P_{n}(z),
\ee
where $\mathrm{v}(z):=-\ln w(z)$ and
\be\label{an}
A_{n}(z):=\frac{1}{h_{n}}\int_{-1}^{1}\frac{\mathrm{v}'(z)-\mathrm{v}'(y)}{z-y}P_{n}^{2}(y)w(y)dy,
\ee
\be\label{bn}
B_{n}(z):=\frac{1}{h_{n-1}}\int_{-1}^{1}\frac{\mathrm{v}'(z)-\mathrm{v}'(y)}{z-y}P_{n}(y)P_{n-1}(y)w(y)dy.
\ee
For brevity, please note that in the above and in what follows we do not display the $t$-dependence of $P_n(x),\;w(x),\;\bt_n$ and $h_n$ unless it is required.

The functions $A_n(z)$ and $B_n(z)$ satisfy the following identities for all $z\in\mathbb{C}\cup\{\infty\}$:
\be
B_{n+1}(z)+B_{n}(z)=z A_{n}(z)-\mathrm{v}'(z), \tag{$S_{1}$}
\ee
\be
1+z(B_{n+1}(z)-B_{n}(z))=\beta_{n+1}A_{n+1}(z)-\beta_{n}A_{n-1}(z), \tag{$S_{2}$}
\ee
\be
B_{n}^{2}(z)+\mathrm{v}'(z)B_{n}(z)+\sum_{j=0}^{n-1}A_{j}(z)=\beta_{n}A_{n}(z)A_{n-1}(z). \tag{$S_{2}'$}
\ee
Here ($S_{2}'$) is obtained from the combination of ($S_{1}$) and ($S_{2}$), and it gives better insight into the recurrence coefficient $\bt_n$. The conditions ($S_{1}$), ($S_{2}$) and ($S_{2}'$) are usually called the supplementary (compatibility) conditions for the ladder operators, which will play an important role in our following analysis.

Furthermore, $P_{n}(z)$ satisfies the second-order linear ordinary differential equation
\be\label{general}
P_{n}''(z)-\left(\mathrm{v}'(z)+\frac{A_{n}'(z)}{A_{n}(z)}\right)P_{n}'(z)+\left(B_{n}'(z)-B_{n}(z)\frac{A_{n}'(z)}{A_{n}(z)}
+\sum_{j=0}^{n-1}A_{j}(z)\right)P_{n}(z)=0,
\ee
which is obtained by eliminating $P_{n-1}(z)$ from (\ref{lowering}) and (\ref{raising}) together with the aid of ($S_{2}'$).

For our weight (\ref{wei}), we have
$$
\mathrm{v}(z)=-\ln w(z)=tz^2-\al\ln(1-z^2).
$$
It follows that
$$
\mathrm{v}'(z)=2tz+\frac{2\al z}{1-z^2}
$$
and
\be\label{vp}
\frac{\mathrm{v}'(z)-\mathrm{v}'(y)}{z-y}=2t+\frac{2\al(1+zy)}{(1-z^2)(1-y^2)}.
\ee
Substituting (\ref{vp}) into the definitions of $A_n(z)$ and $B_n(z)$ in (\ref{an}) and (\ref{bn}) and taking account of the parity, we find
\be\label{anz}
A_{n}(z)=2t+\frac{R_{n}(t)}{1-z^2},
\ee
\be\label{bnz}
B_{n}(z)=\frac{z\:r_{n}(t)}{1-z^2},
\ee
where
\be\label{Rnt}
R_{n}(t):=\frac{2\al}{h_{n}}\int_{-1}^{1}\frac{1}{1-y^2}P_{n}^{2}(y)w(y) dy,
\ee
\be\label{rnt}
r_{n}(t):=\frac{2\al}{h_{n-1}}\int_{-1}^{1}\frac{y}{1-y^2}P_{n}(y)P_{n-1}(y)w(y) dy.
\ee
\begin{remark}
From (\ref{Rnt}) and (\ref{rnt}), we have
\be\label{Rn0}
R_n(0)=2n+1+2\al
\ee
and
\be\label{rn0}
r_n(0)=n,
\ee
which can also be found in \cite{ChenIsmail2005}.
\end{remark}
\begin{proposition}
The auxiliary quantities $R_n(t),\;r_n(t)$ and the recurrence coefficient $\bt_n$ have the following relations:
\be\label{re1}
r_{n+1}(t)+r_{n}(t)=R_n(t)-2\al,
\ee
\be\label{re2}
r_n(t)=n-2t\bt_n,
\ee
\be\label{re3}
r_n^2(t)+2\al\: r_n(t)=\bt_n R_n(t)R_{n-1}(t),
\ee
\be\label{re4}
2(t-\al)r_n(t)-r_n^2(t)+\sum_{j=0}^{n-1}R_j(t)=2t\bt_n(R_n(t)+R_{n-1}(t)).
\ee
\end{proposition}
\begin{proof}
Substituting (\ref{anz}) and (\ref{bnz}) into ($S_{1}$) and equating the coefficients of $\frac{1}{1-z^2}$, we get (\ref{re1}). Similarly, substituting (\ref{anz}) and (\ref{bnz}) into ($S_{2}'$) and equating the constant terms, the coefficients of $\frac{1}{(1-z^2)^2}$ and $\frac{1}{1-z^2}$, we obtain (\ref{re2}), (\ref{re3}) and (\ref{re4}), respectively.
\end{proof}
\begin{remark}
The identity (\ref{re2}) can also be derived from ($S_{2}$).
\end{remark}

\begin{theorem}
The recurrence coefficient $\bt_n$ satisfies the following second-order nonlinear difference equation:
\be\label{btd}
(n-2 t \beta_n)^2+2 \alpha  (n-2 t \beta_n)-\beta_n (2 n-1+2 \alpha -2 t \beta_{n-1}-2 t \beta_n) (2 n+1+2 \alpha -2 t \beta_n-2 t \beta_{n+1})=0,
\ee
with the initial conditions
$$
\bt_0=0,\qquad \bt_1=\frac{\Phi\left(\frac{3}{2},\frac{5}{2}+\al;-t\right)}{(3+2\al)\Phi\left(\frac{1}{2},\frac{3}{2}+\al;-t\right)}
$$
and $\Phi(\cdot, \cdot; \cdot)$ is the Kummer's confluent hypergeometric function.
\end{theorem}
\begin{proof}
The combination of (\ref{re1}) and (\ref{re2}) gives
\be\label{Rnb}
R_n(t)=2 n+1+2 \alpha -2 t \beta_n-2 t \beta_{n+1}.
\ee
Substituting (\ref{re2}) and (\ref{Rnb}) into (\ref{re3}), we obtain (\ref{btd}). The initial conditions $\bt_0$ is given in the introduction, and $\bt_1$ is easily computed from $\bt_1=h_1/h_0=\mu_2(t)/\mu_0(t)$ by using (\ref{mom}).
\end{proof}

\begin{lemma}
We have the identity
\be\label{id3}
\bt_nR_n(t)=r_n(t)+2\mathrm{p}(n,t)+2t\bt_n\bt_{n-1}.
\ee
\end{lemma}
\begin{proof}
The proof will be given in Section 4. See Lemma \ref{lem} and Remark \ref{rem}.
\end{proof}
\begin{theorem}
The sub-leading coefficient of the monic orthogonal polynomials, $\mathrm{p}(n):=\mathrm{p}(n,t)$, satisfies the second-order nonlinear difference equation:
\begin{small}
\bea\label{pnd}
&&\big(n-2 t\: \mathrm{p}(n)+2 t\: \mathrm{p}(n+1)\big)^2+2 \alpha\big(n-2 t\: \mathrm{p}(n)+2 t\: \mathrm{p}(n+1)\big)\nonumber\\
&-&\big(2 n-1+2 \alpha -2 t\: \mathrm{p}(n-1)+2 t\: \mathrm{p}(n+1)\big)\nonumber\\
&\times&\big[n+2 \mathrm{p}(n)+2 t\big(\mathrm{p}(n)-\mathrm{p}(n+1)\big) \big(\mathrm{p}(n-1)-\mathrm{p}(n)-1\big) \big]=0,
\eea
with the initial conditions
$$
\mathrm{p}(1,t)=0,\qquad \mathrm{p}(2,t)=-\frac{\Phi\left(\frac{3}{2},\frac{5}{2}+\al;-t\right)}{(3+2\al)\Phi\left(\frac{1}{2},\frac{3}{2}+\al;-t\right)}.
$$
\end{small}
\end{theorem}
\begin{proof}
Substituting (\ref{id3}) into (\ref{re3}) gives us
$$
r_n^2(t)+2\al\: r_n(t)=(r_n(t)+2\mathrm{p}(n,t)+2t\bt_n\bt_{n-1})R_{n-1}(t).
$$
Using (\ref{re2}) and (\ref{Rnb}) to eliminate $r_n(t)$ and $R_{n-1}(t)$, we have
$$
(n-2t\bt_n)^2+2\al(n-2t\bt_n)=(n-2t\bt_n+2\mathrm{p}(n,t)+2t\bt_n\bt_{n-1})(2 n-1+2 \alpha -2 t \beta_{n-1}-2 t \beta_{n}).
$$
Taking account of (\ref{be1}), we obtain (\ref{pnd}). The initial conditions $\mathrm{p}(1,t)$ is given in the introduction, and $\mathrm{p}(2,t)$ is derived from (\ref{be1}) by using the value of $\bt_1$ in the last theorem.
\end{proof}

We end this section by showing the second-order differential equation for our time-dependent Jacobi polynomials, with the coefficients all expressed in terms of $\bt_n$.
\begin{theorem}
The monic orthogonal polynomials $P_n(z),\; n=0,1,2,\ldots,$ satisfy the second-order differential equation
\be\label{odep}
P_n''(z)-\left(\mathrm{v}'(z)+\frac{A_{n}'(z)}{A_{n}(z)}\right)P_n'(z)+\left(B_{n}'(z)-B_{n}(z)\frac{A_{n}'(z)}{A_{n}(z)}
+\sum_{j=0}^{n-1}A_{j}(z)\right)P_n(z)=0,
\ee
where
\be\label{an1}
A_n(z)=2t+\frac{2n+1+2\al-2t(\bt_n+\bt_{n+1})}{1-z^2},
\ee
\be\label{bn1}
B_n(z)=\frac{z(n-2t\bt_n)}{1-z^2},
\ee
\bea\label{sum1}
\sum_{j=0}^{n-1}A_{j}(z)&=&2nt+\frac{n^2-2n(t-\al)+2t\bt_n(2t+1-2t\bt_{n+1})}{1-z^2}\nonumber\\
&+&\frac{2 t (n-2 t \beta_n) (n+2 \alpha -2 t \beta_n)}{\left[2 n+1+2\alpha-2t(\bt_n+\bt_{n+1})\right](1-z^2)}
\eea
and
$$
\mathrm{v}'(z)=2tz+\frac{2\al z}{1-z^2}.
$$
\end{theorem}
\begin{proof}
We restate (\ref{general}) as (\ref{odep}). Substituting (\ref{Rnb}) into (\ref{anz}) and (\ref{re2}) into (\ref{bnz}), we obtain (\ref{an1}) and (\ref{bn1}) respectively. Moreover, from (\ref{anz}) we have
\bea
\sum_{j=0}^{n-1}A_{j}(z)&=&2nt+\frac{\sum_{j=0}^{n-1}R_j(t)}{1-z^2}\nonumber\\
&=&2nt+\frac{r_n^2(t)-2(t-\al)r_n(t)+2t\bt_nR_n(t)+\frac{2t(r_n^2(t)+2\al\: r_n(t))}{R_n(t)}}{1-z^2},\nonumber
\eea
where use has been made of (\ref{re4}) and (\ref{re3}). Inserting (\ref{re2}) and (\ref{Rnb}) into the above equation, we obtain (\ref{sum1}). This completes the proof.
\end{proof}
\begin{remark}
When $t=0$, equation (\ref{odep}) is reduced to the differential equation satisfied by the classical symmetric Jacobi polynomials \cite[(4.2.1)]{Szego}.
\end{remark}

\section{Asymptotics of the Recurrence Coefficient $\bt_n(t)$ and the Sub-leading Coefficient $\mathrm{p}(n,t)$}
In this section, we derive the full asymptotic expansions for $\bt_n(t)$ and $\mathrm{p}(n,t)$ as $n\rightarrow\infty$ by using the nonlinear second-order difference equations.

Kuijlaars et al. \cite[Theorem 1.10]{Kuij} showed that as $n\rightarrow\infty$, $\bt_n$ has the expansion of the form
\be\label{bex}
\bt_n=a_0+\sum_{j=1}^{\infty}\frac{a_j}{n^j},
\ee
where $a_j, j=0,1,2,\ldots,$ are the expansion coefficients to be determined. The expansion form (\ref{bex}) can also be obtained by the Coulomb fluid approach \cite{ChenIsmail}.
\begin{theorem}\label{thm}
The recurrence coefficient $\bt_n$ has the following large $n$ expansion
\be\label{beta}
\bt_n=\frac{1}{4}+\sum_{j=1}^{\infty}\frac{a_j}{n^j},\qquad\qquad n\rightarrow\infty,
\ee
where the first few expansion coefficients are
$$
a_1=0,
$$
$$
a_2=\frac{1}{16} (1-4 \alpha ^2),
$$
$$
a_3=\frac{1}{16} (1-4 \alpha ^2) (t-2 \alpha ),
$$
$$
a_4=\frac{1}{64} (1-4 \alpha ^2) (3 t^2-12 \alpha  t+12 \alpha ^2+1),
$$
$$
a_5=\frac{1}{64} (1-4 \alpha ^2) \left[2t^3-12 \alpha  t^2+(11+20 \alpha ^2) t-4\al (1+4 \alpha ^2 )\right],
$$
$$
a_6=\frac{1}{256} (1-4 \alpha ^2) \left[5 t^4-40 \alpha  t^3+20(5+4 \alpha ^2) t^2-20 \alpha  (11+4 \alpha ^2) t+80 \alpha ^4+40 \alpha ^2+1\right].
$$
\end{theorem}
\begin{proof}
Substituting (\ref{bex}) into (\ref{btd}), the difference equation satisfied by $\bt_n$, we obtain the following expression by taking a large $n$ limit:
$$
e_{-2}n^2+e_{-1}n+e_0+\sum_{j=1}^{\infty}\frac{e_j}{n^j}=0,
$$
where all the coefficients $e_j, j=-2, -1, 0,\ldots,$ are explicitly known and should be equal to 0 identically.

The equation $e_{-2}=0$ reads
$$
1-4a_0=0.
$$
Then we have
$$
a_0=\frac{1}{4}.
$$
Setting $a_0=\frac{1}{4}$, the equation $e_{-1}=0$ gives
$$
-4a_1=0,
$$
and we get
$$
a_1=0.
$$
With the values of $a_0$ and $a_1$, the equation $e_0=0$ gives us
$$
\frac{1}{4}(1-4\al^2-16a_2)=0,
$$
and we have
$$
a_2=\frac{1}{16}(1-4 \alpha ^2).
$$
With the above $a_0$, $a_1$ and $a_2$, the equation $e_1=0$ shows
$$
\frac{1}{4}\left[(1-4\al^2)(t-2\al)-16a_3\right]=0,
$$
and we obtain
$$
a_3=\frac{1}{16} (1-4 \alpha ^2) (t-2 \alpha ).
$$
Proceeding with this procedure, we can easily obtain the values of higher order expansion coefficients $a_4, a_5, a_6, \ldots.$ This establishes the theorem.
\end{proof}
\begin{remark}
When $t=0$, it was shown in \cite[(2.18)]{ChenIsmail2005} that
\be\label{btc}
\bt_n(0)=\frac{n (n+2 \alpha)}{(2 n-1+2 \alpha) (2 n+1+2 \alpha)},
\ee
and then as $n\rightarrow\infty$,
\bea\label{bt0}
\bt_n(0)&=&\frac{1}{4}+\frac{1-4\al^2}{16n^2}-\frac{\al(1-4\al^2)}{8n^3}+\frac{(1-4\al^2)(1+12\al^2)}{64n^4}-\frac{\al(1-4\al^2)(1+4\al^2)}{16n^5}\nonumber\\[8pt]
&+&\frac{(1-4\al^2)(1+40\al^2+80\al^4)}{256n^6}+O(n^{-7}).
\eea
It is easy to check that our result in Theorem \ref{thm} agrees with (\ref{bt0}) when $t=0$.
\end{remark}
\begin{remark}
It can be seen that for the case $\al=\pm\frac{1}{2}$, all coefficients $a_j,\;j=1, 2, \ldots$ in (\ref{beta}) vanish. This fact has been pointed out by Kuijlaars et al. \cite[p. 346]{Kuij} and in that case one can prove that $\bt_n=\frac{1}{4}+O(\mathrm{e}^{-cn})$ for some $c>0$; see also \cite[p. 196]{Kuij1}.
\end{remark}
\begin{remark}
The difference equation approach has also been used to obtain the large $n$ asymptotic expansions for $\bt_n$ in the symmetric Pollaczek-Jacobi type weight \cite{Min2021} and the generalized Freud weight problems \cite{Clarkson1,Clarkson2}.
\end{remark}
\begin{theorem}\label{thmp}
The sub-leading coefficient $\mathrm{p}(n,t)$ has the following expansion as $n\rightarrow\infty$:
\be\label{pne}
\mathrm{p}(n,t)=b_{-1}n+b_0+\sum_{j=1}^{\infty}\frac{b_j}{n^j},
\ee
where the first few coefficients are
$$
b_{-1}=-\frac{1}{4},
$$
$$
b_0=\frac{1}{16} (t+2+4 \alpha),
$$
$$
b_1=\frac{1}{16}(1-4 \alpha ^2),
$$
$$
b_2=\frac{1}{32} (1-4 \alpha ^2) (t+1-2 \alpha),
$$
$$
b_3=\frac{1}{64} (1-4 \alpha ^2) (t+1-2 \alpha)^2,
$$
$$
b_4=\frac{1}{256} (1-4 \alpha ^2) \left[ 2 t^3+6 (1-2 \alpha) t^2+(20 \alpha ^2-24 \alpha+15) t+2 (1-2 \alpha)^3\right],
$$
$$
b_5=\frac{1}{256} (1-4 \alpha ^2) \left[t^4+4(1-2 \alpha ) t^3+8 (2 \alpha ^2-3 \alpha +3) t^2 +2 (1-2 \alpha) (4 \alpha ^2-8 \alpha +11)t+(1-2 \alpha )^4\right].
$$
\end{theorem}

\begin{proof}
From (\ref{id3}) we have
$$
2\mathrm{p}(n,t)=\bt_nR_n(t)-r_n(t)-2t\bt_n\bt_{n-1}.
$$
Inserting (\ref{Rnb}) and (\ref{re2}) into the above gives the expression of $\mathrm{p}(n,t)$ in terms of $\bt_n$:
\be\label{pb}
\mathrm{p}(n,t)=\bt_n\left(n+\al+\frac{1}{2}+t-t\bt_n-t\bt_{n-1}-t\bt_{n+1}\right)-\frac{n}{2}.
\ee
Substituting the expansion of $\bt_n$ in Theorem \ref{thm} into (\ref{pb}), we obtain the desired result by taking a large $n$ limit.
\end{proof}
\begin{remark}
If we suppose that $\mathrm{p}(n,t)$ has the large $n$ expansion of the form
$$
\mathrm{p}(n,t)=b_{-1}n+b_0+\sum_{j=1}^{\infty}\frac{b_j}{n^j}.
$$
Substituting it into (\ref{pnd}), the difference equation satisfied by $\mathrm{p}(n,t)$, we obtain the same result in Theorem \ref{thmp} by adopting the similar procedure in the proof of Theorem \ref{thm}.
\end{remark}
\begin{remark}
For consistency, substituting (\ref{pne}) into $\bt_n=\mathrm{p}(n,t)-\mathrm{p}(n+1,t)$ and taking a large $n$ limit, we find that the result agrees completely with the large $n$ expansion for $\bt_n$ obtained in Theorem \ref{thm}.
\end{remark}
\begin{remark}
When $t=0$, our monic orthogonal polynomials $P_n(x)$ are related to the Gegenbauer (ultraspherical) polynomials $C_{n}^{\nu}(x)$ as follows:
$$
P_n(x)=\frac{n!\Gamma\big(\al+\frac{1}{2}\big)}{2^n\Gamma\big(n+\al+\frac{1}{2}\big)}C_{n}^{\al+\frac{1}{2}}(x),
$$
where
$$
C_{n}^{\nu}(x)=\sum_{k=0}^{[n/2]}\frac{(-1)^k 2^{n-2k}\Gamma(n+\nu-k)}{k!(n-2k)!\Gamma(\nu)}x^{n-2k}
$$
and $[\lambda]$ denotes the largest integer $\leq \lambda$ \cite[p. 125]{Lebedev}. See also \cite[p. 94--98]{Ismail}. It follows that
$$
\mathrm{p}(n,0)=\mathrm{p}(n,t)|_{t=0}=-\frac{n(n-1)}{2(2n-1+2\al)}.
$$
Then as $n\rightarrow\infty$,
\bea
\mathrm{p}(n,0)&=&-\frac{n}{4}+\frac{1+2 \alpha}{8}+\frac{(1-2 \alpha ) (1+2 \alpha)}{16 n}+\frac{(1-2 \alpha )^2 (1+2 \alpha)}{32 n^2}+\frac{(1-2 \alpha)^3 (1+2 \alpha)}{64 n^3}\nonumber\\[8pt]
&+&\frac{(1-2 \alpha )^4 (1+2 \alpha)}{128 n^4}+\frac{(1-2 \alpha)^5 (1+2 \alpha)}{256 n^5}+O\left(\frac{1}{n^6}\right).\nonumber
\eea
It is easy to see that this result coincides with Theorem \ref{thmp} when $t=0$.
\end{remark}

\section{$t$ Evolution and Painlev\'{e} V}
In this section, we are going to study the evolution of auxiliary quantities in $t$. We will show that the auxiliary quantities $R_n(t)$ and $r_n(t)$ satisfy the coupled Riccati equations. This allows us to establish the relation of our problem with the Painlev\'{e} equations. Moreover, we also obtain the second-order differential equation satisfied by the recurrence coefficient $\bt_n(t)$.

\begin{lemma}\label{lem}
We have
\be\label{id1}
2t\frac{d}{dt}\ln h_n(t)=R_n(t)-2n-1-2\al,
\ee
\be\label{id2}
2t\frac{d}{dt}\mathrm{p}(n,t)=\bt_nR_n(t)-r_n(t)-2\mathrm{p}(n,t).
\ee
\end{lemma}
\begin{proof}
From (\ref{or}) we have
$$
h_{n}(t)=\int_{-1}^{1}P_{n}^2(x,t)(1-x^2)^\al\mathrm{e}^{-tx^2}dx,\qquad n=0,1,2,\ldots.
$$
Taking a derivative with respect to $t$ and multiplying by $2t$ on both sides give us
$$
2t\: h_n'(t)=-2t\int_{-1}^{1}x^2P_{n}^2(x,t)(1-x^2)^\al\mathrm{e}^{-tx^2}dx.
$$
Through integration by parts, and with the aid of the orthogonality relation, we get
$$
2t\: h_n'(t)=-(2n+1+2\al)h_n+2\al\int_{-1}^{1}\frac{1}{1-x^2}P_{n}^2(x,t)(1-x^2)^{\al}\mathrm{e}^{-tx^2}dx.
$$
In view of the definition of $R_n(t)$ in (\ref{Rnt}), we obtain (\ref{id1}).

To proceed, from (\ref{or}) we also have
$$
\int_{-1}^{1}P_{n}(x,t)P_{n-2}(x,t)(1-x^2)^\al\mathrm{e}^{-tx^2}dx=0,\qquad n=1,2,\ldots.
$$
Differentiating the above formula with respect to $t$, we obtain
\be\label{pn}
\frac{d}{dt}\mathrm{p}(n,t)=\frac{1}{h_{n-2}}\int_{-1}^{1}x^2P_{n}(x,t)P_{n-2}(x,t)(1-x^2)^\al\mathrm{e}^{-tx^2}dx.
\ee
Using (\ref{rr}) and (\ref{be2}), we have
\be\label{rr1}
\frac{P_{n-2}(x,t)}{h_{n-2}}=\frac{xP_{n-1}(x,t)}{h_{n-1}}-\frac{P_{n}(x,t)}{h_{n-1}}.
\ee
Inserting (\ref{rr1}) into (\ref{pn}) gives
\bea\label{pn1}
\frac{d}{dt}\mathrm{p}(n,t)&=&\frac{1}{h_{n-1}}\int_{-1}^{1}x^3P_{n}(x,t)P_{n-1}(x,t)(1-x^2)^\al\mathrm{e}^{-tx^2}dx\nonumber\\
&-&\frac{1}{h_{n-1}}\int_{-1}^{1}x^2P_{n}^2(x,t)(1-x^2)^\al\mathrm{e}^{-tx^2}dx.
\eea
By integration by parts and using the formula
$$
\frac{1}{h_{n-1}}\int_{-1}^{1}x^2\frac{d}{dx}P_{n}(x,t)\:P_{n-1}(x,t)(1-x^2)^\al\mathrm{e}^{-tx^2}dx=n\bt_n-2\mathrm{p}(n,t),
$$
we obtain (\ref{id2}) after some elementary computations.
\end{proof}
\begin{remark}\label{rem}
From (\ref{pn}) we also have
\be\label{pnt}
\frac{d}{dt}\mathrm{p}(n,t)=\frac{h_n}{h_{n-2}}=\bt_n\bt_{n-1}.
\ee
Substituting it into (\ref{id2}) produces an identity
$$
\bt_nR_n(t)=r_n(t)+2\mathrm{p}(n,t)+2t\bt_n\bt_{n-1}.
$$
\end{remark}

\begin{proposition}
The auxiliary quantities $R_n(t)$ and $r_n(t)$ satisfy the coupled Riccati equations:
\be\label{ri1}
2t\:r_n'(t)=\frac{2t(r_n^2(t)+2\al\: r_n(t))}{R_n(t)}-(n-r_n(t))R_n(t),
\ee
\be\label{ri2}
2tR_n'(t)=4\al t-R_n^2(t)+(2\al+1-2t)R_n(t)+2(2t+R_n(t))r_n(t).
\ee
\end{proposition}
\begin{proof}
In view of (\ref{be2}), we have from (\ref{id1}) that
$$
2t\frac{d}{dt}\ln\bt_n(t)=R_n(t)-R_{n-1}(t)-2.
$$
That is,
\be\label{bp}
2t\bt_n'(t)=\bt_n(R_n(t)-R_{n-1}(t)-2).
\ee
With the aid of (\ref{re2}) and (\ref{re3}), we obtain
\be\label{bp1}
4t^2\bt_n'(t)=(n-r_n(t))(R_n(t)-2)-\frac{2t(r_n^2(t)+2\al\:r_n(t))}{R_n(t)}.
\ee
On the other hand, taking a derivative in (\ref{re2}) gives
\be\label{de}
r_n'(t)=-2\bt_n(t)-2t\bt_n'(t).
\ee
Then
\bea\label{bp2}
2t^2\bt_n'(t)&=&-2t\bt_n(t)-t\: r_n'(t)\nonumber\\
&=&r_n(t)-t\: r_n'(t)-n.
\eea
Substituting (\ref{bp2}) into (\ref{bp1}), we obtain equation (\ref{ri1}).

Next, taking a derivative in (\ref{id3}) and with the aid of (\ref{pnt}), we find
\be\label{eq1}
\bt_n'(t)R_n(t)+\bt_n(t)R_n'(t)=r_n'(t)+2\bt_n(t)\bt_{n-1}(t)+(2t\bt_n(t)\bt_{n-1}(t))'.
\ee
From (\ref{re2}) and (\ref{re1}) we have
$$
2t\bt_{n-1}(t)=n-1-r_{n-1}(t)=n-1+2\al+r_n(t)-R_{n-1}(t).
$$
It follows that
\be\label{bb1}
2t\bt_n(t)\bt_{n-1}(t)=(n-1+2\al+r_n(t))\bt_n(t)-\frac{r_n^2(t)+2\al\:r_n(t)}{R_n(t)}
\ee
and
\be\label{bb2}
4t^2\bt_n(t)\bt_{n-1}(t)=(n-1+2\al+r_n(t))(n-r_n(t))-\frac{2t(r_n^2(t)+2\al\:r_n(t))}{R_n(t)},
\ee
where use has been made of (\ref{re3}) and (\ref{re2}).

Taking a derivative in (\ref{bb1}) and inserting it into (\ref{eq1}), and then multiplying by $4t^2$ on both sides, we find
\bea
4t^2\bt_n'(t)R_n(t)+4t^2\bt_n(t)R_n'(t)&=&4t^2r_n'(t)+8t^2\bt_n(t)\bt_{n-1}(t)+4t^2r_n'(t)\bt_n(t)\nonumber\\
&+&(n-1+2\al+r_n(t))4t^2\bt_n'(t)-4t^2\left(\frac{r_n^2(t)+2\al\:r_n(t)}{R_n(t)}\right)'.\nonumber
\eea
Substituting (\ref{bp1}) and (\ref{bb2}) into the above equation, and then using (\ref{re2}) to eliminate $\bt_n(t)$, we get an equation satisfied by $r_n(t),\; r_n'(t),\; R_n(t)$ and $R_n'(t)$. Finally by making use of (\ref{ri1}) to eliminate $r_n'(t)$, we obtain the following equation after some simplifications:
\bea
&&\left[n R_n^2(t)-2 t\: r_n^2(t)-(R_n^2(t)+4 \alpha  t)r_n(t) \right]\nonumber\\
&\times&\left[2 t R_n'(t)+R_n^2(t)+ (2 t-1-2\al)R_n(t)-2 r_n(t) (R_n(t)+2 t)-4 \alpha  t\right]=0.\nonumber
\eea
Obviously, this leads to two equations, one of which is an algebraic equation and does not hold. So we should discard this algebraic equation and then we arrive at (\ref{ri2}). The proof is complete.
\end{proof}

\begin{theorem}
Let
\be\label{tr}
W_n(t):=1+\frac{2t}{R_n(t)}.
\ee
Then $W_n(t)$ satisfies the Painlev\'{e} V equation \cite{Gromak}
\be\label{pv}
W_{n}''=\frac{(3 W_n-1) (W_n')^2}{2W_n (W_n-1) }-\frac{W_n'}{t}+\frac{(W_n-1)^2 }{t^2}\left(\mu_1 W_n +\frac{\mu_2}{W_n}\right)+\frac{\mu_3 W_n}{t}+\frac{\mu_4 W_n(W_n+1) }{W_n-1},
\ee
with the parameters
$$
\mu_{1}=\frac{\alpha ^2}{2},\qquad \mu_2=-\frac{1}{8},\qquad \mu_3=\frac{1}{2} (2 n+1+2 \alpha),\qquad \mu_4=-\frac{1}{2}.
$$
The initial conditions are
$$
W_n(0)=1,\qquad W_n'(0)=\frac{2}{2n+1+2\al}.
$$
\end{theorem}
\begin{proof}
Eliminating $r_n(t)$ from the coupled Riccati equations (\ref{ri1}) and (\ref{ri2}), we obtain the second-order differential equation satisfied by $R_n(t)$:
\bea
&&4 t^2 R_n (R_n+2 t)R_n''-4 t^2 (R_n+t)(R_n')^2+4 t R_n^2 R_n'-R_n^5+(2n+1+2 \alpha-5 t)R_n^4\nonumber\\
&+&4t(2n+1+2 \alpha-2 t)R_n^3+t \left[4 \alpha ^2+3-4 t^2+4t (2 n+1+2 \alpha)\right]R_n^2+16 \alpha ^2 t^2 R_n+16 \alpha ^2 t^3=0.\nonumber
\eea
With the transformation (\ref{tr}), we arrive at (\ref{pv}). The initial conditions follow from (\ref{Rn0}).
\end{proof}

\begin{theorem}\label{thmb}
The recurrence coefficient $\bt_n(t)$ satisfies the following second-order nonlinear ordinary differential equation
\begin{small}
\bea\label{btde}
&&\big\{8 t^2 ( t \beta_n'+ \beta_n) (t \beta_n ''+2\beta_n ')+4t( 2 t-2n+1-2 \alpha+4t\bt_n) ( t \beta_n'+ \beta_n)^2-4 t( t \beta_n'+ \beta_n)\nonumber\\
&\times& \left[2 n\alpha +2( n-2 \alpha)  (n-2 t \beta_n)-3 (n-2 t \beta_n)^2\right]+4 (n-2 t \beta_n)^4-4  (n-3 \alpha+t)(n-2 t \beta_n)^3\nonumber\\
&+&4 \left[n (t-3 \alpha )-2 \alpha  (t-\alpha)\right](n-2 t \beta_n)^2+8 n\alpha  (t-\alpha )(n-2 t \beta_n)\big\}^2\nonumber\\
&=&\left[2 t ( t \beta_n ''+2 \beta_n ')+(2t-2\al+1-2r_n) ( t \beta_n'+ \beta_n)+3 (n-2 t \beta_n)^2-2 ( n-2 \alpha ) (n-2 t \beta_n)-2 n\alpha\right]^2\nonumber\\
&\times&16t^2\left[ \beta_n(n-2 t \beta_n) (n+2 \alpha -2 t \beta_n)+ ( t \beta_n'+ \beta_n)^2\right],
\eea
with the initial conditions
$$
\bt_n(0)=\frac{n(n+2\al)}{(2n-1+2\al)(2n+1+2\al)},\\[8pt]
$$
$$
\bt_n'(0)=\frac{4n(n+\al)(n+2\al)(1-4\al^2)}{(2n-1+2\al)^2(2n+1+2\al)^2(2n-3+2\al)(2n+3+2\al)}.
$$
\end{small}
\end{theorem}
\begin{proof}
Solving for $R_n(t)$ from equation (\ref{ri1}), we get two solutions. Substituting either solution into (\ref{ri2}) and then removing the square root, we obtain the second-order differential equation satisfied by $r_n(t)$:
\begin{small}
\bea
&&\big\{2 t^2 r_n' r_n''+t(2 t+1-2 \alpha-2 r_n) (r_n')^2+2 t \left[2 n\alpha +2( n-2 \alpha)  r_n-3 r_n^2\right]r_n'+4 r_n^4\nonumber\\
&-&4  (n-3 \alpha+t)r_n^3+4 \left[n (t-3 \alpha )-2 \alpha  (t-\alpha)\right]r_n^2+8 n\alpha  (t-\alpha )r_n\big\}^2\nonumber\\
&=&t\left[2 r_n (n-r_n) (2 \alpha +r_n)+t (r_n')^2\right]\left[2 t r_n''+(2t-2\al+1-2r_n) r_n'-6 r_n^2+4 ( n-2 \alpha ) r_n+4 n\alpha\right]^2.\nonumber
\eea
\end{small}
In view of the relation (\ref{re2}), we obtain (\ref{btde}). The value of $\bt_n(0)$ is given by (\ref{btc}), and $\bt_n'(0)$ follows from (\ref{be1}) and (\ref{pnt}).
\end{proof}

\section{The Hankel Determinant and its Asymptotics}
In this section, we consider the Hankel determinant $D_n(t)$ for the semi-classical Jacobi weight. We first define a quantity
\be\label{def}
H_n(t):=\sum_{j=0}^{n-1}R_j(t).
\ee
From (\ref{hankel}) and (\ref{id1}) we have
\bea
2t\frac{d}{dt}\ln D_n(t)&=&\sum_{j=0}^{n-1}2t\frac{d}{dt}\ln h_j(t)\nonumber\\
&=&\sum_{j=0}^{n-1}R_j(t)-n(n+2\al).\nonumber
\eea
Hence, $H_n(t)$ is related to the logarithmic derivative of the Hankel determinant as follows:
\be\label{hn}
H_n(t)=n(n+2\al)+2t\frac{d}{dt}\ln D_n(t).
\ee

In random matrix theory, the logarithmic derivatives of the Hankel determinants are very important, and they usually satisfy certain nonlinear second-order differential and difference equations. Sometimes these equations can be transformed to the Jimbo-Miwa-Okamoto $\s$-forms of the Painlev\'{e} equations and its discrete $\s$-forms. See \cite{BCE,ChenIts,Min2019,Min2020} for reference.
\begin{theorem}
The quantity $H_n(t)$ satisfies the following second-order differential equation
\begin{small}
\bea\label{ode}
&&\big\{4 t^3 (H_n'')^2+4 t^2 H_n'' \left[H_n'-2 (t+\alpha)\right]+8 t^2 (H_n')^3-t (H_n')^2 \left[4 H_n+4 (t+\al)^2-32nt-1\right]\nonumber\\
&-&4tH_n' \left[4  (2n+\alpha+t)H_n+ (3t+\alpha) (4 nt +4n\alpha+1)\right]+8 (n+\alpha+t) H_n^2\nonumber\\
&+&4H_n \left[2t^3+6 t^2 (n+\alpha)+ t (8 n\alpha+6 \alpha ^2+1)+2 \alpha ^2 (n+\alpha)\right]+8 t (t+\alpha)^2 \left[2 n(t+\al)+1\right]\big\}^2\nonumber\\
&=&16 \left[H_n-2 t H_n'+(t+\alpha)^2\right]\big\{2 t^2 H_n''+t (4 H_n+8 n t+1) H_n'-2 H_n^2-2 H_n \left[2 n t+(t+\alpha)^2\right]\nonumber\\
&-&2 t (t+\alpha) \left[2 n (t+\alpha)+1\right]\big\}^2,
\eea
with the initial conditions
$$
H_n(0)=n(n+2\al),\qquad H_n'(0)=-\frac{2n(2n^2+2n\al-1)}{(2n-1+2\al)(2n+1+2\al)}.
$$
\end{small}
\end{theorem}
\begin{proof}
With the definition (\ref{def}), we write equation (\ref{re4}) as
\be\label{equ1}
2t\bt_nR_n(t)+2t\bt_n R_{n-1}(t)=2(t-\al)r_n(t)-r_n^2(t)+H_n(t).
\ee
On the other hand, from (\ref{bp}) we have
\be\label{equ2}
2t\bt_nR_n(t)-2t\bt_n R_{n-1}(t)=4t\bt_n+4t^2\bt_n'=- 2t\: r_n'(t).
\ee
where use has been made of (\ref{de}) in the second equality.

The sum and difference of (\ref{equ1}) and (\ref{equ2}) give us
\be\label{equ3}
4t\bt_nR_n(t)=2(t-\al)r_n(t)-r_n^2(t)+H_n(t)- 2t\: r_n'(t)
\ee
and
\be\label{equ4}
4t\bt_nR_{n-1}(t)=2(t-\al)r_n(t)-r_n^2(t)+H_n(t)+ 2t\: r_n'(t),
\ee
respectively.
The product of (\ref{equ3}) and (\ref{equ4}) produces
\be\label{pro}
8t(n-r_n(t))(r_n^2(t)+2\al r_n(t))=\left[2(t-\al)r_n(t)-r_n^2(t)+H_n(t)\right]^2-4t^2 (r_n'(t))^2,
\ee
where we have used (\ref{re3}) and (\ref{re2}).

Next, replacing $n$ by $j$ in (\ref{re1}) and taking a sum from $j=0$ to $j=n-1$, we have
$$
2\sum_{j=0}^{n-1}r_j(t)+r_n(t)=\sum_{j=0}^{n-1}R_j(t)-2n\al.
$$
Taking account of (\ref{re2}) and (\ref{def}), it follows that
$$
n(n-1)-4t\sum_{j=0}^{n-1}\bt_j+r_n(t)=H_n(t)-2n\al.
$$
Using (\ref{sum}), we obtain
\be\label{pnr}
4t\:\mathrm{p}(n,t)=H_n(t)-r_n(t)-n(n-1+2\al).
\ee
Taking a derivative with respect to $t$ gives
\be\label{eq2}
4\mathrm{p}(n,t)+4t\frac{d}{dt}\mathrm{p}(n,t)=H_n'(t)-r_n'(t).
\ee

On the other hand, we have from (\ref{id2}) that
\be\label{eq3}
4\mathrm{p}(n,t)+4t\frac{d}{dt}\mathrm{p}(n,t)=2\bt_nR_n(t)-2r_n(t).
\ee
The combination of (\ref{eq2}) and (\ref{eq3}) shows
$$
2\bt_nR_n(t)=H_n'(t)+2r_n(t)-r_n'(t).
$$
Substituting it into (\ref{equ3}), we obtain
\be\label{eq4}
r_n^2(t)+2(t+\al)r_n(t)+2t H_n'(t)-H_n(t)=0.
\ee
Note that the terms involving $r_n'(t)$ disappear!
Then, equation (\ref{eq4}) can be viewed as a quadratic equation for $r_n(t)$, which has two solutions
$$
r_n(t)=-t-\alpha \pm\sqrt{(t+\al)^2+H_n(t)-2 t H_n'(t)}.
$$
Substituting either solution into (\ref{pro}) and clearing the square root, we obtain equation (\ref{ode}) after some simplifications. The value of $H_n(0)$ follows from (\ref{def}) or (\ref{hn}); $H_n'(0)$ is from the combination of (\ref{def}), (\ref{re1}) and (\ref{re2}) with the aid of the values of $\bt_n(0)$ and $\bt_n'(0)$ in Theorem \ref{thmb}.
\end{proof}
\begin{remark}
Although we can not transform equation (\ref{ode}) into the $\s$-form of a certain Painlev\'{e} equation directly, we suppose that $H_n(t)$ can be expressed as the sum of two $\s$-functions of Painlev\'{e} V with different parameters. This phenomenon always occurs in the symmetric weight problems; see \cite{Basor2015,Min2020b,Min2021,MLC} for reference.
\end{remark}

\begin{theorem}
The quantity $H_n(t)$ satisfies the following second-order difference equation
\begin{small}
\bea\label{hnd}
0&=&(2t+H_n-H_{n-1})^2(2t+H_{n+1}-H_{n})^2(H_n+nH_{n-1}-nH_{n+1})\nonumber\\
&+&\big[n(2t+H_n-H_{n-1})(2t+H_{n+1}-H_{n})-2t(2nt+H_n)\big]\nonumber\\
&\times&\big[(2t+H_n-H_{n-1})(2t+H_{n+1}-H_{n})(2t-n-2\al+H_{n+1}-H_{n-1})+2t(2nt+H_n)\big],
\eea
with the initial conditions
$$
H_1(t)=\frac{(1+2\al)\Phi\left(\frac{1}{2},\frac{1}{2}+\al;-t\right)}{\Phi\left(\frac{1}{2},\frac{3}{2}+\al;-t\right)},\\[8pt]
$$
$$
H_2(t)=\frac{(1+2\al)\Phi\left(\frac{1}{2},\frac{1}{2}+\al;-t\right)}{\Phi\left(\frac{1}{2},\frac{3}{2}+\al;-t\right)}
+\frac{(3+2\al)\Phi\left(\frac{3}{2},\frac{3}{2}+\al;-t\right)}{\Phi\left(\frac{3}{2},\frac{5}{2}+\al;-t\right)}.
$$
\end{small}
\end{theorem}
\begin{proof}
From (\ref{def}) we have
\be\label{r1}
R_n(t)=H_{n+1}(t)-H_n(t)
\ee
and
\be\label{r2}
R_{n-1}(t)=H_{n}(t)-H_{n-1}(t).
\ee
Taking a sum of (\ref{re3}) and (\ref{re4}) to eliminate $r_n^2(t)$, and in view of (\ref{def}), we get
$$
2t\: r_n(t)+H_n(t)=\bt_nR_n(t)R_{n-1}(t)+2t\bt_n(R_n(t)+R_{n-1}(t)).
$$
Multiplying by $2t$ on both sides and eliminating $\bt_n$ with the aid of (\ref{re2}) give us
$$
2t(2t\: r_n(t)+H_n(t))=(n-r_n(t))R_n(t)R_{n-1}(t)+2t(n-r_n(t))(R_n(t)+R_{n-1}(t)).
$$
Substituting (\ref{r1}) and (\ref{r2}) into the above equation, we can then express $r_n(t)$ in terms of $H_{n+1}(t), H_n(t)$ and $H_{n-1}(t)$:
\be\label{rne}
r_n(t)=\frac{n(2t+H_{n}(t)-H_{n-1}(t))(2t+H_{n+1}(t)-H_n(t))-2t(2nt+H_n(t))}{(2t+H_{n}(t)-H_{n-1}(t))(2t+H_{n+1}(t)-H_n(t))}.
\ee

To proceed, multiplying by $2t$ on both sides of (\ref{id3}) and using (\ref{re2}) to eliminate $\bt_n$, together with the aid of (\ref{r1}), we obtain
\be\label{equa}
(n-r_n(t))(H_{n+1}(t)-H_n(t))=2t\:r_n(t)+4t\:\mathrm{p}(n,t)+(n-r_n(t))(n-1-r_{n-1}(t)).
\ee
From (\ref{re1}) we have
\bea\label{rnr}
r_{n-1}(t)&=&R_{n-1}(t)-2\al-r_n(t)\nonumber\\
&=&H_{n}(t)-H_{n-1}(t)-2\al-r_n(t).
\eea
Substituting (\ref{pnr}) and (\ref{rnr}) into (\ref{equa}), we obtain the following equation
$$
H_n+n H_{n-1}(t)-n H_{n+1}(t)+r_n(t)\big(2t-2\al-r_n(t)+H_{n+1}(t)-H_{n-1}(t)\big)=0.
$$
Inserting (\ref{rne}) into the above equation, we obtain (\ref{hnd}). The initial conditions follow from (\ref{def}) and (\ref{Rnt}).
\end{proof}
In the end, we show the large $n$ asymptotics of the Hankel determinant $D_n(t)$ in the following theorem.
\begin{theorem}\label{thm1}
The ratio $D_n(t)/D_n(0)$ has the following expansion as $n\rightarrow\infty$:
\bea\label{ra}
\frac{D_n(t)}{D_n(0)}&=&\exp\bigg[-\frac{nt}{2}+\frac{t(t+8\al)}{16}+\frac{t(1-4\al^2)}{8n}+\frac{t(1-4\al^2)(t-4\al)}{32n^2}\nonumber\\[10pt]
&+&\frac{t(1-4\al^2)(t^2-6\al t+12\al^2+3)}{96n^3}+O\left(\frac{1}{n^4}\right)\bigg].
\eea
Moreover, the Hankel determinant $D_n(t)$ has the large $n$ asymptotic expansion
\bea\label{dnt}
D_n(t)&=&\frac{\pi^{n+\al+\frac{1}{2}}n^{\al^2-\frac{1}{4}}G^2\big(\frac{1}{2}\big)}{2^{n^2+(2\al-1)(n+\al)}G^2(\al+1)}\times\exp\bigg[-\frac{n t}{2}+\frac{t (t+8 \alpha)}{16}+\frac{(1-4 \alpha ^2) (t-2 \alpha )}{8 n}\nonumber\\[10pt]
&+&\frac{(1-4 \alpha ^2) (6 t^2-24 \alpha  t+28 \alpha ^2-3)}{192 n^2}+\frac{(1-4 \alpha ^2) \big(t^3-6 \alpha  t^2+3(1+4 \alpha ^2) t+3 \alpha(1-4\al^2)\big)}{96 n^3}\nonumber\\[10pt]
&+&O\left(\frac{1}{n^4}\right)\bigg],
\eea
where $G(\cdot)$ is the Barnes $G$-function.
\end{theorem}
\begin{proof}
From (\ref{pnr}) we have
\bea
H_n(t)&=&n(n-1+2\al)+r_n(t)+4t\:\mathrm{p}(n,t)\nonumber\\
&=&n(n+2\al)-2t\bt_n+4t\:\mathrm{p}(n,t)\label{bpe}\\
&=&n(n+2\al)+2t\:\mathrm{p}(n,t)+2t\:\mathrm{p}(n+1,t),\label{hn1}
\eea
where use has been made of (\ref{re2}) and (\ref{be1}). Comparing (\ref{hn1}) with (\ref{hn}) gives us
$$
\frac{d}{dt}\ln D_n(t)=\mathrm{p}(n,t)+\mathrm{p}(n+1,t).
$$
It follows that
$$
\ln \frac{D_n(t)}{D_n(0)}=\int_{0}^{t}\big(\mathrm{p}(n,s)+\mathrm{p}(n+1,s)\big)ds.
$$
Substituting (\ref{pne}) into the above equation and taking a large $n$ limit, we obtain (\ref{ra}).

Next, it is known that the Barnes $G$-function has the following asymptotic expansion \cite[p. 285]{Barnes},
\bea
\log G(z+1)&=&z^2\left(\frac{\log z}{2}-\frac{3}{4}\right)+\frac{z}{2}\log(2\pi)-\frac{\log z}{12}+\frac{1}{12}-\log A-\frac{1}{240 z^2}+\frac{1}{1008 z^4}\nonumber\\
&-&\frac{1}{1440 z^6}+\frac{1}{1056 z^8}+O\left(\frac{1}{z^{10}}\right),\qquad z\rightarrow+\infty,\nonumber
\eea
where $A$ is the Glaisher-Kinkelin constant ($A=1.2824271291...$), and
$$
G\left(\frac{1}{2}\right)=2^{\frac{1}{24}}\pi^{-\frac{1}{4}}\mathrm{e}^{\frac{1}{8}}A^{-\frac{3}{2}}.
$$
Then, we obtain the large $n$ asymptotics of $D_n(0)$ from (\ref{bg}):
\bea\label{dn0}
D_n(0)&=&\frac{\pi^{n+\al+\frac{1}{2}}n^{\al^2-\frac{1}{4}}G^2\big(\frac{1}{2}\big)}{2^{n^2+(2\al-1)(n+\al)}G^2(\al+1)}\times\exp\bigg[-\frac{\al(1-4\alpha ^2)}{4n}+\frac{(1-4\al^2)(28\al^2-3)}{192 n^2}\nonumber\\[10pt]
&+&\frac{\alpha  (1-4 \alpha ^2)^2}{32 n^3}+O\left(\frac{1}{n^4}\right)\bigg].
\eea
The combination of (\ref{ra}) and (\ref{dn0}) gives (\ref{dnt}). The proof is complete.
\end{proof}
\begin{remark}We provide an alternate proof of Theorem \ref{thm1} in the following.
Substituting (\ref{pb}) into (\ref{bpe}), $H_n(t)$ can be expressed in terms of $\bt_n$:
\be\label{hnt}
H_n(t)=n(n+2\al-2t)+4t\bt_n(n+\al+t-t\bt_n-t\bt_{n-1}-t\bt_{n+1}).
\ee
The combination of (\ref{hn}) and (\ref{hnt}) gives
$$
\frac{d}{dt}\ln D_n(t)=2\bt_n(n+\al+t-t\bt_n-t\bt_{n-1}-t\bt_{n+1})-n.
$$
It follows that
\be\label{dnt0}
\ln \frac{D_n(t)}{D_n(0)}=\int_{0}^{t}\left[2\bt_n(s)\big(n+\al+s-s\bt_n(s)-s\bt_{n-1}(s)-s\bt_{n+1}(s)\big)-n\right]ds.
\ee
Substituting the expansion of $\bt_n$ in Theorem \ref{thm} into (\ref{dnt0}), we obtain the same result in Theorem \ref{thm1} by taking a large $n$ limit.
\end{remark}
\begin{remark}
Recently, Charlier and Gharakhloo \cite{CG} studied the large $n$ asymptotics of Hankel determinants of Laguerre and Jacobi type weights with Fisher-Hartwig singularities by using the Riemann-Hilbert approach. Our result in Theorem \ref{thm1} is consistent with \cite[Theorem 1.3]{CG}. We thank the reviewer for pointing this out. However, we compute the higher order terms of $n$ up to $O(n^{-3})$, which are not given in \cite[Theorem 1.3]{CG}. In fact, we can compute the higher order terms of $n$ up to any degree by using our approach.
\end{remark}

\section*{Acknowledgments}
The authors would like to thank the referee for helpful comments which substantially improved the presentation of this paper.
The work of Chao Min was partially supported by the National Natural Science Foundation of China under grant number 12001212, by the Fundamental Research Funds for the Central Universities under grant number ZQN-902 and by the Scientific Research Funds of Huaqiao University under grant number 17BS402. The work of
Yang Chen was partially supported by the Macau Science and Technology Development Fund under grant number FDCT 023/2017/A1 and by the University of Macau under grant number MYRG 2018-00125-FST.

\section*{Conflicts of Interest}
The authors have no conflicts of interest to declare that are relevant to the content of this article.

\section*{Data Availability Statements}
Data sharing not applicable to this article as no datasets were generated or analysed during the current study.

\end{document}